\documentclass[12pt,leqno]{amsart}
\usepackage{amsmath,color,amssymb}
\usepackage{graphicx}
\newcommand{\R}{{\mathbb R}}
\newcommand{\Z}{{\mathbb Z}}
\newcommand{\N}{{\mathbb N}}

\newcommand{\Var}{\operatorname{Var}}
\newcommand{\Cov}{\operatorname{Cov}}

\newtheorem{lemma}{Lemma}[section]

\newtheorem{theorem}[lemma]{Theorem}
\newtheorem{proposition}[lemma]{Proposition}
\newtheorem{definition}[lemma]{Definition}

\newtheorem{remark}[lemma]{Remark}

\begin{document}
\title[Change-Point Detection Based on Two-Sample U-Statistics]
{Change-Point Detection under Dependence Based on Two-Sample U-Statistics}
\author[H. Dehling]{Herold Dehling}
\author[R. Fried]{Roland Fried}
\author[I. Garcia]{Isabel Garcia}
\author[M. Wendler]{Martin Wendler}
\today

\address{
Fakult\"at f\"ur Mathematik, Ruhr-Universit\"at Bo\-chum, 
44780 Bochum, Germany}
\email{herold.dehling@rub.de}
\address{Fakult\"at f\"ur Statistik, Technische Universit\"at Dortmund,
44221 Dortmund, Germany}
\email{fried@statistik.tu-dortmund.de}
\address{Fakult\"at f\"ur Mathematik, Ruhr-Universit\"at Bo\-chum, 
44780 Bochum, Germany}
\email{isabel.garciaarboleda@rub.de}
\address{Fakult\"at f\"ur Mathematik, Ruhr-Universit\"at Bo\-chum, 
44780 Bochum, Germany}
\email{martin.wendler@rub.de}

\thanks{This research was supported  by the 
Collaborative Research Grant 823, Project C3 {\em Analysis of Structural 
Change in Dynamic Processes}, of the German Research 
Foundation.}
\begin{abstract}
We study the detection of change-points in time series. The classical CUSUM statistic for detection of jumps in the mean is known to be sensitive to outliers. We thus propose a robust test based on the Wilcoxon
two-sample test statistic. The asymptotic distribution of this test can be derived from a functional central limit theorem for two-sample U-statistics. We extend a theorem of Cs\"org\H{o} and Horv\'ath to the case of
dependent data.
\end{abstract}
\maketitle
\tableofcontents

\section{Introduction}
Change-point tests address the question whether a stochastic process is stationary during the entire observation period or not. In the case of independent data, there is a well-developed theory; see the book by Cs\"org\H{o} and Horv\'ath (1997) for an excellent survey. When the data are dependent, much less is known. The CUSUM statistic has been intensely studied, even for dependent data; see again Cs\"org\H{o} and Horv\'ath (1997). The CUSUM test, however, is not robust against outliers in the data. In the present paper, we study a robust test which is based on the two-sample Wilcoxon test statistic. Simulations show that this test outperforms the CUSUM test in the case of heavy-tailed data. 

In order to derive the asymptotic distribution of the test, we study the stochastic process
\begin{equation}
 \sum_{i=1}^{[n\lambda]} \sum_{j=[n\lambda]+1}^n h(X_i,X_j),\; 0\leq \lambda \leq 1,
\label{eq:us-process}
\end{equation}
where 
$h:\R^2\rightarrow \R$ is a kernel function. In the case of independent data, the asymptotic distribution 
of this process has been studied by Cs\"org\H{o} and Horv\'ath (1988). In the present paper, we extend their result to short range dependent data $(X_i)_{i\geq 1}$. Similar results have been obtained for long range dependent data by Dehling, Rooch and Taqqu (2012), albeit with completely different methods.

U-statistics have been introduced by Hoeffding (1948), where the asymptotic normality was established both for the one-sample as well as the two-sample U-statistic in the case of independent data.  The asymptotic distribution of one-sample U-statistics of dependent data was studied by Sen (1963, 1972),
Yoshihara (1976), Denker and Keller (1983, 1985) and by Borovkova, Burton and Dehling (2001) in the so-called non-degenerate case, and by
Babbel (1989) and Leucht (2012) in the degenerate case. For two-sample U-statistics, Dehling and Fried (2012) established the asymptotic normality of 
$\sum_{i=1}^{n_1} \sum_{j=n_1+1}^{n_1+n_2} h(X_i,X_j)$ for dependent data, when $n_1,n_2\rightarrow \infty$. The main theoretical result of the present paper is a functional version of this limit theorem.

In our paper, we focus on data that can be represented as functionals of a mixing process. In this way, we cover most examples from time series analysis, such as ARMA and ARCH processes, but also data from chaotic dynamical systems. For a survey of processes that have a representation as functional of a mixing process, see e.g. Borovkova, Burton and Dehling (2001). Earlier references can be found in Ibragimov and Linnik (1970) and Billingsley (1968).

\section{Definitions and Main Results}

Given the samples $X_1,\ldots,X_m$ and $Y_1,\ldots,Y_n$, and a kernel $h(x,y)$, we define the two-sample U-statistic
\[
  U_{n_1,n_2}:=\frac{1}{n_1\, n_2} \sum_{i=1}^{n_1} \sum_{j=1}^{n_2} h(X_i,Y_j).
\]
More generally, one can define U-statistics with multivariate kernels $h(x_1,\ldots,x_k,y_1,\ldots,y_l)$. In the present paper, for the ease of exposition, we will restrict attention to bivariate kernels $h(x,y)$. The main results, however, can easily be extended to the multivariate case. 

Assuming that $(X_i)_{i\geq 1}$ and $(Y_i)_{i\geq 1}$ are stationary processes with one-dimensional marginal distribution functions $F$ and $G$, respectively, we can test the hypothesis $H:\, F=G$ using the two-sample U-statistic. E.g., the kernel $h(x,y)=y-x$ leads to the U-statistic 
\[
 U_{n_1,n_2} =\frac{1}{n_1\, n_2} \sum_{i=1}^{n_1}\sum_{j=1}^{n_2} (Y_j-X_i)
 =\frac{1}{n_2} \sum_{j=1}^{n_2} Y_j -\frac{1}{n_1} \sum_{i=1}^{n_1} X_i,
\]
and thus to the familiar two-sample Gau\ss -test. Similarly, the kernel $h(x,y)=1_{\{x\leq y \}}$ leads to the U-statistic
\[
 U_{n_1,n_2}=\frac{1}{n_1\, n_2} \sum_{i=1}^{n_1} \sum_{j=1}^{n_2} 1_{\{X_i\leq X_j  \}},
\]
and thus to the 2-sample Mann-Whitney-Wilcoxon test.

In the present paper, we investigate tests for a change-point in the mean of a stochastic process 
$(X_i)_{i\geq 1}$. We consider the model
\[
 X_i=\mu_i+\xi_i, \; i\geq 1,
\]
where $(\mu_i)_{i\geq 1}$ are unknown constants and where $(\xi_i)_{i\geq 1}$ is a stochastic process. We want to test the hypothesis 
\[
 H:\; \mu_1=\ldots=\mu_n
\]
against the alternative
\[
 A:\, \mbox{ There exists } 1\leq k\leq n-1 \mbox{ such that } \mu_1=\ldots =\mu_k\neq \mu_{k+1}
 =\ldots=\mu_n.
\]
Tests for the change-point problem are often derived from 2-sample tests applied to the samples 
$X_1,\ldots,X_k$ and $X_{k+1},\ldots,X_n$, for all possible $1\leq k\leq n-1$. For two-sample tests based on U-statistics with kernel $h(x,y)$, this leads to the test statistic 
$\sum_{i=1}^k \sum_{j=k+1}^n h(X_i,X_j)$, $1\leq k\leq n$, and thus to the processes 
\begin{equation}
  U_n(\lambda)=\sum_{i=1}^{[n\lambda]} \sum_{j=[n\lambda]+1}^n h(X_i,X_j),
 \; 0\leq \lambda \leq 1.
\label{eq:u-proc}
\end{equation}
In this paper, we will derive a functional limit theorem for the processes 
$(U_n(\lambda))_{0\leq \lambda \leq 1}$. Specifically, we will show that under certain technical assumptions on the kernel $h$ and on the process $(X_i)_{i\geq 1}$, a properly centered and renormalized version of 
$(U_n(\lambda))_{0\leq \lambda \leq 1}$ converges to a Gaussian process.

In our paper, we will assume that the process $(\xi_i)_{i\geq 0}$ is weakly dependent. More specifically, we will assume that $(\xi_i)_{i\geq 0}$ can be represented as a functional of an absolutely regular process. 

\begin{definition}
(i) Given a stochastic process $(X_n)_{n\in \Z}$ we denote by $\mathcal{A}^{k}_{l}$  the $\sigma-$algebra generated by $(X_{k},\ldots,X_{l})$. The process is called absolutely regular if 
\begin{equation}
\beta(k)=\sup_{n}\left\{\sup\sum^{J}_{j=1}\sum^{I}_{i=1}|P(A_{i}\cap B_{j})-P(A_{i})P(B_{j})|\right\}
\rightarrow 0,
\end{equation}
where the last supremum is over all finite $\mathcal{A}^{n}_{1}-$measurable partitions $(A_{1},\ldots,A_{I})$ and all finite $\mathcal{A}^{\infty}_{n+k}-$measurable partitions $(B_{1},\ldots,B_{J}).$
\\[1mm]
(ii) The process is called strongly mixing if
\begin{equation}
\alpha(k)=\sup\left\{\left|P(A\cap B)-P(A)P(B)\right|\ \big|A\in\mathcal{A}^{n}_{1},\ B\in\mathcal{A}^{\infty}_{n+k},\ n\in \N \right\}\rightarrow0.
\end{equation}
(iii) The process $(X_{n})_{n\geq 1}$ is called a two-sided functional of an absolutely regular sequence if there exists an absolutely regular process $(Z_{n})_{n\in \Z}$ and a measurable function $f:\R^\Z\rightarrow\R$ such that 
\[
 X_{i}=f((Z_{i+n})_{n\in\Z}).
\] 
Analogously, $(X_{n})_{n\geq 1}$ is called a one-sided functional  if $X_{i}=f((Z_{i+n})_{n\geq 0})$.
\\[1mm]
(iv) The process $(X_{n})_{n\geq 1}$ is called $1$-approximating functional with coefficients $(a_{k})_{k\geq 1}$ if 
\begin{equation}
E\left|X_{i}-E(X_{i}|Z_{i-k},\ldots,Z_{i+k})\right|\leq a_{k}
\end{equation}

\end{definition}

In addition to weak dependence conditions on the process $(X_i)_{i\geq 1}$, the asymptotic analysis of the process (\ref{eq:u-proc}) requires some continuity assumptions on the kernel functions $h(x,y)$. 
We use the notion of $1$-continuity, which was introduced by Borovkova, Burton and Dehling (2001). Alternative continuity conditions have been used by Denker and Keller (1986).

\begin{definition}
The kernel $h(x,y)$ is called $1$-continuous,  if there exists a function $\phi:(0,\infty)\rightarrow (0,\infty)$ with $\phi(\epsilon)=o(1)$ as $\epsilon\rightarrow 0$ such that for all $\epsilon>0$ 
\begin{align}
E(|h(X',Y)-h(X,Y)|_{\{|X-X'|\leq\epsilon\}})\leq\phi(\epsilon)\\
E(|h(X,Y')-h(X,Y)|_{\{|Y-Y'|\leq\epsilon\}})\leq\phi(\epsilon)
\end{align}
for all random variables $X,X',Y$ and $Y'$ having the same marginal distribution as $X$.
\end{definition}

The most important technical tool in the study of U-statistics is Hoeffding's decomposition, originally introduced by Hoeffding (1948). We write
\begin{equation}
 h(x,y)=\theta + h_1(x) +h_2(y) +g(x,y),
\label{eq:h-decomp}
\end{equation}
where the terms on the right-hand side are defined as follows:
\begin{eqnarray*}
\theta &=& Eh(X,Y) \\
h_1(x) &=& Eh(x,Y) -\theta \\
h_2(y) &=& Eh(X,y) -\theta\\
g(x,y) &=& h(x,y)-h_1(x) -h_2(y) -\theta.
\end{eqnarray*}
Here, $X$ and $Y$ are two independent random variables with the same distribution as $X_1$.
Observe that, by Fubini's theorem,
\[
  E(h_1(X))=E(h_2(X))=0.
\]
In addition, the kernel $g(x,y)$ is degenerate in the sense of the following definition.

\begin{definition}
Let $(X_i)_{i\geq 1}$ be a stationary process, and let $g(x,y)$ be a measurable function. We say that $g(x,y)$ is degenerate if
\begin{equation}
E(g(x,X_1))=E(g(X_1,y))=0,
\label{eq:deg}
\end{equation}
for all $x,y\in \R$.
\end{definition}

The following theorem, a functional central limit theorem for two-sample $U$-statistics of dependent data, is the main theoretical result of the present paper.

\begin{theorem}\label{LimitThm}
Let $(X_{n})_{n\geq 1}$ be a $1$-approximating functional with constants $(a_{k})_{k\geq 1}$ of an absolutely regular process with mixing coefficients $(\beta(k))_{k\geq 1}$, satisfying
\begin{equation}
\sum^{\infty}_{k=1}k^2(\beta(k)+\sqrt{a_{k}}+\phi(a_{k}))<\infty,
\label{eq:mix-cond}
\end{equation}
and let $h(x,y)$ be a $1$-continuous bounded kernel.
Then, as $n\rightarrow \infty$, the $D[0,1]$-valued process 
\begin{equation}
T_{n}(\lambda):=\frac{1}{n^{3/2}}\sum^{[\lambda n]}_{i=1}\sum^{n}_{j=[\lambda n]+1}
(h(X_{i},X_{j}) -\theta),\; 0\leq \lambda \leq 1,
\end{equation}
converges in distribution towards a mean-zero Gaussian processes with representation
\begin{equation}
 Z(\lambda)=(1-\lambda)W_{1}(\lambda)+\lambda(W_{2}(1)-W_{2}(\lambda)),\; 0\leq \lambda \leq 1,
\label{eq:z-repr}
\end{equation}
where $(W_{1}(\lambda),W_{2}(\lambda))_{0\leq \lambda \leq 1}$ is a two-dimensional Brownian motion with mean zero and covariance function $\Cov(W_k(s),W_l(t)) = \min(s,t) \sigma_{kl}$, where 
\begin{equation}\label{variance}
\sigma_{kl}=E(h_{k}(X_{0})h_{l}(X_{0}))+2\, \sum_{j=1}^\infty \Cov(h_{k}(X_{0}),h_{l}(X_{j})),\; k,l=1,2.
\end{equation}
\end{theorem}
\begin{remark}{\rm
(i) In the case of i.i.d.\ data, Theorem~\ref{LimitThm} was established by Cs\"org\H{o} and Horv\'ath (1988). In the case of long-range dependent data, weak convergence of the process $(T_n(\lambda))_{0\leq \lambda \leq 1}$ has been studied by Dehling, Rooch and Taqqu (2013) and by Rooch (2012), albeit with a normalization different from $n^{3/2}$. 
\\[1mm]
(ii) Using the representation (\ref{eq:z-repr}), one can calculate the autocovariance function of the process $(Z(\lambda))_{0\leq \lambda \leq 1}$. We obtain 
\begin{multline}
\Cov(Z(\lambda),Z(\mu))=\sigma_{11}[(1-\lambda)(1-\mu)\min\{\lambda,\mu\}]+\sigma_{22}[\lambda\mu(1-\mu-\lambda+\min\{\lambda,\mu\})]\\
+\sigma_{12}[\mu(1-\lambda)(\lambda-\min\{\lambda,\mu\})+\lambda(1-\mu)(\mu-\min\{\lambda,\mu\})].
\end{multline}
(iii) For the kernel $h(x,y)=y-x$, we can analyze the asymptotic behavior of the process $T_n(\lambda)$ 
using the functional central limit theorem (FCLT). Note that, since $X_j-X_i=(X_j-E(X_j))-(X_i-E(X_i))$, we may assume without loss of generality that $X_i$ has mean zero. Then we get the representation
\begin{eqnarray}
T_n(\lambda) &=& \frac{1}{n^{3/2}} \sum_{i=1}^{[n\lambda]} \sum_{j=[n\lambda]+1}^n (X_j-X_i) 
 \nonumber \\
&=& \frac{[n\lambda]}{n}\frac{1}{\sqrt{n}} \sum_{i=1}^n X_i -\frac{1}{\sqrt{n}} \sum_{i=1}^{[n\lambda]} X_i.
\end{eqnarray}
Thus, weak convergence of $(T_n(\lambda))_{0\leq \lambda \leq 1}$ can be derived from the FCLT for the partial sum process $\frac{1}{\sqrt{n}} \sum_{i=1}^{[n\lambda]} X_i$. Such FCLTs have been proved under a wide range of conditions, e.g. for functionals of absolutely regular data.
}
\end{remark}

We finally want to state an important special case of Theorem~\ref{LimitThm}, namely when the kernel is anti-symmetric, i.e. when $h(x,y)=-h(y,x)$. Kernels that occur in  connection with change-point tests usually have this property. For anti-symmetric kernels, the limit process has a much simpler structure; moreover
one can give a simpler direct proof in this case.

\begin{theorem}
Let $(X_{n})_{n\geq 1}$ be a $1$-approximating functional with constants $(a_{k})_{k\geq 1}$ of an absolutely regular process with mixing coefficients $(\beta(k))_{k\geq 1}$, satisfying (\ref{eq:mix-cond}),
and let $h(x,y)$ be a $1$-continuous bounded anti-symmetric kernel.
Then, as $n\rightarrow \infty$, the $D[0,1]$-valued process 
\begin{equation}
T_{n}(\lambda):=\frac{1}{n^{3/2}}\sum^{[\lambda n]}_{i=1}\sum^{n}_{j=[\lambda n]+1}
(h(X_{i},X_{j}) -\theta),\; 0\leq \lambda \leq 1,
\end{equation}
converges in distribution towards the mean-zero Gaussian process
$ \sigma\, W^{(0)} (\lambda),\; 0 \leq \lambda \leq 1$,
where $(W^{0}(\lambda))_{0\leq \lambda \leq 1}$ is a standard Brownian bridge and 
\begin{equation}\label{eq:anti-symm-var}
\sigma^{2} = \Var(h_1(X_1)) +2\sum_{i=2}^\infty \Cov(h_1(X_1),h_1(X_k))
\end{equation}
\label{th:anti-symm}
\end{theorem}

\section{Application to Change Point Problems}
In this section, we will apply Theorem~\ref{LimitThm} in order to derive the asymptotic distribution of two change-point test statistics. Specifically, we  wish to test the null hypothesis 
\begin{equation}
\label{H0}
H_{0}:\mu_{1}=\ldots=\mu_{n}   
\end{equation}
against the alternative of a level shift at an unknown point in time, i.e.
\begin{equation}
\label{H1}
H_{A}:\mu_{1}=\ldots=\mu_k\neq \mu_{k+1}=\ldots=\mu_{n}, \mbox{ for some } k\in \{1,\ldots,n-1\}.
\end{equation}
We consider the following two test statistics,
\begin{eqnarray}
T_{1,n}&=&\max_{1\leq k<n}\left|\frac{1}{ n^{3/2}}\sum^{k}_{i=1}\sum^{n}_{j=k+1} \left(1_{\{X_{i}<X_{j}\}}-1/2\right)\right|
\label{eq:t1n}\\
T_{2,n}&=&\max_{1\leq k<n}\left|\frac{1}{n^{3/2}}\sum^{k}_{i=1}\sum^{n}_{j=k+1}\left(X_{i}-X_{j}\right)\right|.
\label{eq:t2n}
\end{eqnarray}

\begin{theorem}
Let $(X_{n})_{n\geq 1}$ be a $1$-approximating functional with constants $(a_{k})_{k\geq 1}$ of an absolutely regular process with mixing coefficients $(\beta(k))_{k\geq 1}$, satisfying (\ref{eq:mix-cond}), and assume that $X_1$ has a  distribution function $F(x)$ with bounded density.
Then, under the null hypothesis $H_0$, 
\begin{eqnarray}
\label{UstaW}
T_{1,n}&\rightarrow & \sigma_1 \sup_{0\leq\lambda\leq1}|W^{(0)}(\lambda)| \\
\label{UstaM}
T_{2,n}&\rightarrow & \sigma_2 \sup_{0\leq\lambda\leq1}|W^{(0)}(\lambda)|,
\end{eqnarray}
where $(W^{(0)}(\lambda))_{0\leq \leq \lambda \leq 1}$ denotes the standard Brownian bridge process, and
where
\begin{eqnarray}
 \sigma_1^2&=&\Var(F(X_1))+2\, \sum_{k=2}^\infty \Cov(F(X_1),F(X_k))
\label{eq:var-1} \\
 \sigma_2^2&=&\Var(X_1)+2\, \sum_{k=2}^\infty \Cov(X_1,X_k). 
\label{eq:var-2}
\end{eqnarray}
\label{th:CP-test}
\end{theorem}
{\em Proof.} We will establish weak convergence of $T_{1,n}$. In order to do so, we will apply Theorem~\ref{LimitThm} to the kernel $h(x,y)=1_{\{x< y\}}$. Borovkova, Burton and Dehling (2001) showed that this kernel is $1$-continous. By continuity of the distribution function of $X_1$, we get that 
$\theta=P(X<Y)=1/2$. Moreover, we get
\begin{eqnarray*} 
h_1(x)&=& P(x<X_1)-\frac{1}{2}=\frac{1}{2}-F(x) \\
h_2(x)&=& P(X_1 <x) -\frac{1}{2} =F(x)-\frac{1}{2}.
\end{eqnarray*}
Note that $h_2(x) = -h_1(x)$. Hence $W_2(\lambda)=-W_1(\lambda)$, and thus the limit process in Theorem~\ref{LimitThm} has the representation 
\[
 Z(\lambda)=(1-\lambda) W_1(\lambda) +\lambda (W_2(1)-W_2(\lambda))=W_1(\lambda)-\lambda W_1(1).
\]
Here $W_1(\lambda)$ is a Brownian motion with variance $\sigma_1^2$. Weak convergence of $T_{2,n}$ can be shown directly from the functional central limit theorem for the partial sum process; see e.g. Billingsley (1968).
 \hfill $\Box$

\begin{remark}{\rm
(i) The distribution of $\sup_{0\leq \lambda \leq 1} |W(\lambda)|$ is the well-known Kolmogorov-Smirnov distribution. Quantiles of the Kolmogorov-Smirnov distribution can be found in most statistical tables.
\\
(ii) In order to apply Theorem~\ref{th:CP-test}, we need to estimate the variances
$\sigma^2_{1}$ and $\sigma_2^2$.
Regarding $\sigma_2^2$ given in expression (\ref{eq:var-2}), we apply the non-overlapping subsampling estimator
\begin{equation}
\hat{\sigma}_{2}^2=\frac{1}{[n/l_{n}]}\sum^{[n/l_{n}]}_{i=1}\frac{1}{{l_{n}}}\left(\sum^{il_{n}}_{j=(i-1)l_{n}+1}
X_{j}-\frac{l_{n}}{n}\sum^{n}_{j=1} X_{j}\right)^2
\end{equation}
investigated by Carlstein (1986) for $\alpha$-mixing data. In case of AR(1)-processes,
 Carlstein derives
\begin{equation} l_n=\max(\lceil n^{1/3}(2\rho/(1-\rho^2))^{2/3}\rceil,1)\label{eq:adaptblocks}\end{equation}
 as the choice of the block length which minimizes the MSE asymptotically, with $\rho$ being the autocorrelation coefficient at lag 1. \\
Regarding $\sigma_1^2$ given in (\ref{eq:var-1}), one faces the additional challenge that the distribution function $F$ is unknown. This problem has been addressed, e.g. in Dehling, Fried, Sharipov, Vogel and Wornowizki (2013), for the case of functionals of
absolutely regular processes and $F$ being estimated by the empirical distribution function $F_n$. The authors find the subsampling estimator
for $\sigma_1^2$ 
\begin{equation}
\hat{\sigma}_{1}=\frac{1}{[n/l_{n}]}\sqrt{\frac{\pi}{2}}\sum^{[n/l_{n}]}_{i=1}\frac{1}{\sqrt{l_{n}}}\left|\sum^{il_{n}}_{j=(i-1)l_{n}+1} F_n(X_{j})-\frac{l_{n}}{n}\sum^{n}_{j=1} F_n(X_{j})\right|
\end{equation}
employing non-overlapping subsampling to give smaller biases, but somewhat larger MSEs than the corresponding overlapping subsampling estimator. The adaptive choice of the block length $l_n$ proposed by Carlstein worked well in their simulations if the data were generated from a stationary ARMA(1,1) model and an estimate of $\rho$ was plugged in. In the next section, we will explore this and other proposals in situations with level shifts and normally or heavy-tailed innovations.
 }
\end{remark}

\section{Simulation Results}\label{ss31}

The assumptions regarding the underlying process $(X_i)$ in Theorem~\ref{LimitThm} are satisfied by a wide range of  time series, such as AR and ARMA models. To illustrate the results and to investigate the finite sample behavior and the power of the tests based on $T_{1,n}$ and $T_{2,n}$, we will give some simulation results. We study the underlying change-point model
\begin{equation}
 X_{i} = \left\{
\begin{array}{cl}
 \xi_{i}&\mbox{if } i=1,\ldots,[n\lambda]\\
\mu+\xi_{i} &\mbox{if }  i=[n\lambda]+1,\ldots,n.
\end{array}\right.
\end{equation}
Within this model, the hypothesis of no change is equivalent to $\mu=0$.
We assume that the noise follows an AR(1) process, i.e. that
\begin{equation}
\xi_i=\rho\, \xi_{i-1}+\epsilon_i,
\label{eq:AR-noise}
\end{equation}
where  $-1<\rho<1$, and where the innovations $\xi_{i}$ are  i.i.d.  random variables with mean zero.
The innovations $\xi_{i}$ are generated from a standard normal or a $t_\nu$-distribution with $\nu=3$ degrees of freedom, scaled to have the same 84.13\% percentile as the standard normal, which is 1. 
 The autoregression coefficient is varied in $\rho=\{0.0, 0.4, 0.8\}$, corresponding to zero, moderate or strong positive autocorrelation, and the sample size is $n=200$. For the choice of the block length we used Carlstein's adaptive rule outlined above, or a fixed block length
 of $l_n=9$, which is in good agreement with the empirical findings of Dehling et al. (2013) for larger sample sizes and their theoretical result that $l_n$ should be chosen as $o(\sqrt{n})$ to achieve consistency. For the reason of comparison we also included tests employing 
 overlapping subsampling for estimation of the asymptotical variance, applying the same block lengths as the non-overlapping versions. 
 
   Table \ref{tab:size} contains the empirical levels (i.e. the fraction of rejections) of the tests with an asymptotical level of 5\%, obtained from 4000 simulation runs for each situation.
    Note that the tests developed under the assumption of independence, which do not adjust for autocorrelation, become strongly oversized with an increasingly positive autocorrelation, i.e. they reject a true null hypothesis by far too often and are practically useless already for $\rho=0.4$. The performance of the adjusted tests is much better in this respect and in a good agreement with the asymptotical results. Only if the autocorrelation is strong ($\rho=0.8$), the tests with a fixed block length become somewhat anti-conservative (oversized), and even more so for the CUSUM-test. Longer block lengths are needed for stronger positive autocorrelations, and Carlstein's adaptive block length \eqref{eq:adaptblocks} adjusts for this. There is little difference between the tests employing overlapping and non-overlapping subsampling here.

\begin{table}[htb]
\begin{center}
\begin{tabular}{cc|rrrrr|rrrrr}
  &    &  \multicolumn{5}{c|}{$T_{1,n}$} & \multicolumn{5}{c}{$T_{2,n}$}\\ \hline
 && unadj. &  \multicolumn{2}{c}{$l_n$ fixed} & \multicolumn{2}{c|}{adaptive} & unadj. & \multicolumn{2}{c}{$l_n$ fixed} & \multicolumn{2}{c}{adaptive} \\
 $\nu$ &   $\rho$ && ol & nol   & ol & nol&& ol & nol& ol & nol\\ \hline
$\infty$ & 0.0 & 2.8 & 2.0 & 2.9 & 2.0 & 2.2 & 4.5& 2.9& 3.9& 3.7& 3.8\\
$\infty$ & 0.4 & 24.5 & 2.5 & 3.1 & 3.5 & 3.9 &34.2&  3.9&  4.9&  5.5&  6.0\\
$\infty$ & 0.8 &81.6 & 6.2 & 6.5 & 1.9&  2.5& 91.5& 10.5& 10.6&  3.4&  4.0\\
3 & 0.0 & 3.1 &2.2& 2.9 & 2.2 & 2.9& 3.8 & 2.5 & 3.5 & 3.1 &3.1\\
3 & 0.4 & 26.9 & 2.4& 3.0& 3.2& 3.0& 32.0 & 3.3& 3.8 &4.3 &4.9\\
3 & 0.8 & 82.7& 6.9  &7.0&  2.0&  2.8& 90.6 & 10.2& 10.5&  3.2 & 3.9\\
\end{tabular}
\\[2mm]
\caption[Level of Change-Point Tests]{\label{tab:size} Empirical level of the tests based on $T_{1,n}$ and $T_{2,n}$, for $n=200$, with fixed or adaptive subsampling block length $l_n$ and overlapping (ol) or non-overlapping (nol) subsampling. The results are for AR(1) observations with different lag-one autocorrelations $\rho$ and different $t_\nu$-distributed innovations, and based on 4000 simulation runs each.}
\end{center}
\end{table}

In order to investigate the powers of the tests under the alternative, a change in the mean, we consider shifts of increasing height
$\mu$, generating 400 data sets for each situation.
The sample size is again $n=200$ and the change point is after observation number $\tau=[\lambda n]=100$.

Figure \ref{fig:arnormpower} illustrates the powers of the different versions of the tests in case of Gaussian or $t_3$-distributed innovations and several autocorrelation coefficients $\rho$. Under normality, the CUSUM test $T_{2,n}$ is somewhat more powerful than the test $T_{1,n}$ based on the Wilcoxon statistic, while under the $t_3$-distribution it is the other way round. 
The CUSUM test with the fixed block length considered here becomes strongly oversized if $\rho$ is large, while this effect is less severe for the test based on the Wilcoxon statistic.
Carlstein's adaptive choice of the block length increases the power if $\rho$ is small and improves the size of the test substantially if $\rho$ is large. The tests employing overlapping subsampling (not shown here) perform even slightly more powerful in case of zero or moderate autocorrelations, but much less powerful in case of strong autocorrelations.

\begin{figure}[htbp!]\centering
\caption{\label{fig:arnormpower}Power of the tests in case of a shift in the middle of
an AR(1) process with Gaussian (left) and $t_3$-innovations (right) and different lag one correlations
$\rho=0.0$ (top), $\rho=0.4$ (middle) or $\rho=0.8$ (bottom), $n=200$.
 Wilcoxon test $T_{n,1}$ (bold lines) and CUSUM test $T_{n,2}$ (thin lines).
 Adjustment by non-overlapping subsampling with fixed (black) or adaptive block length (grey).  }
\resizebox{15cm}{7cm}{\includegraphics{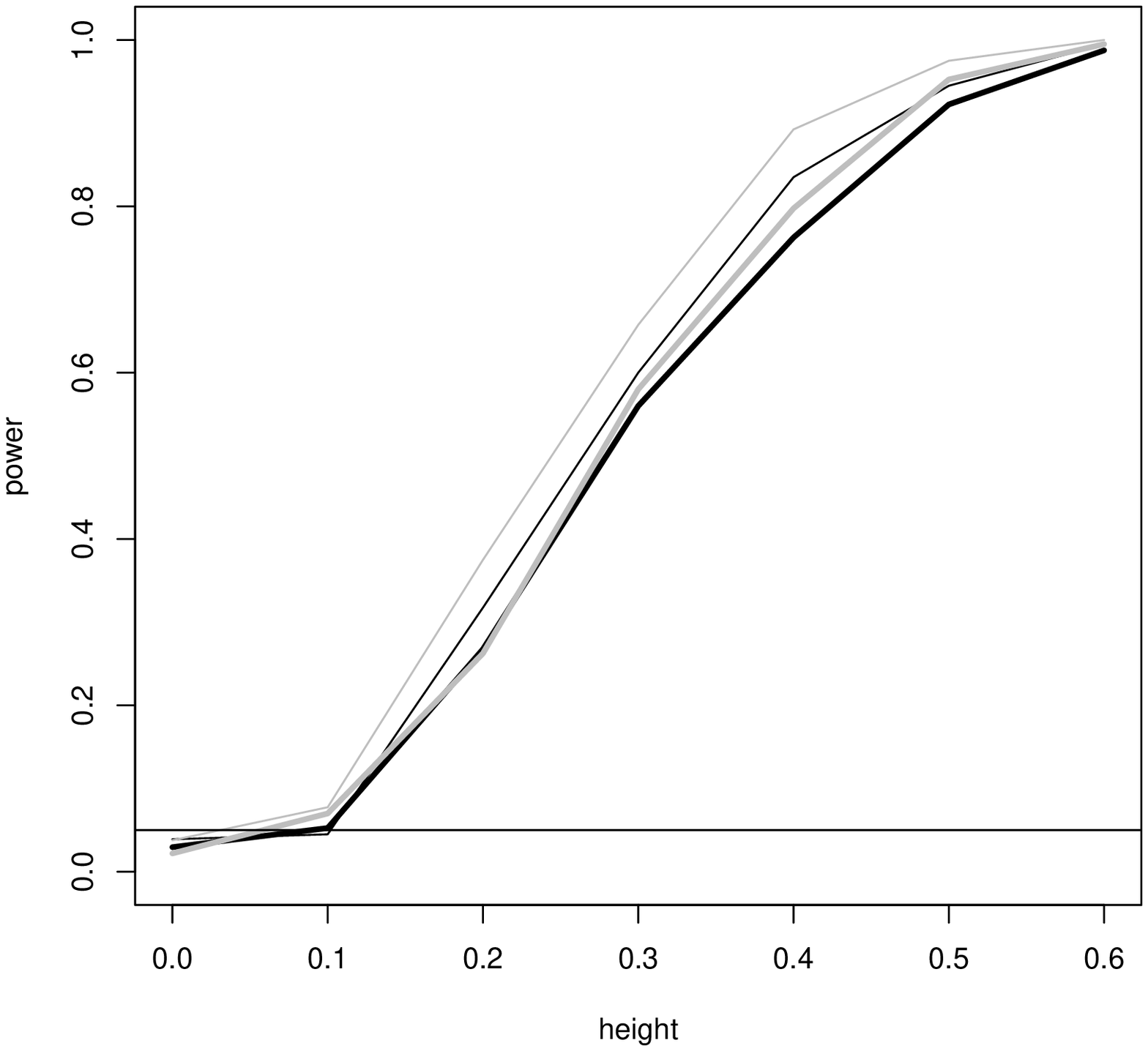}
\includegraphics{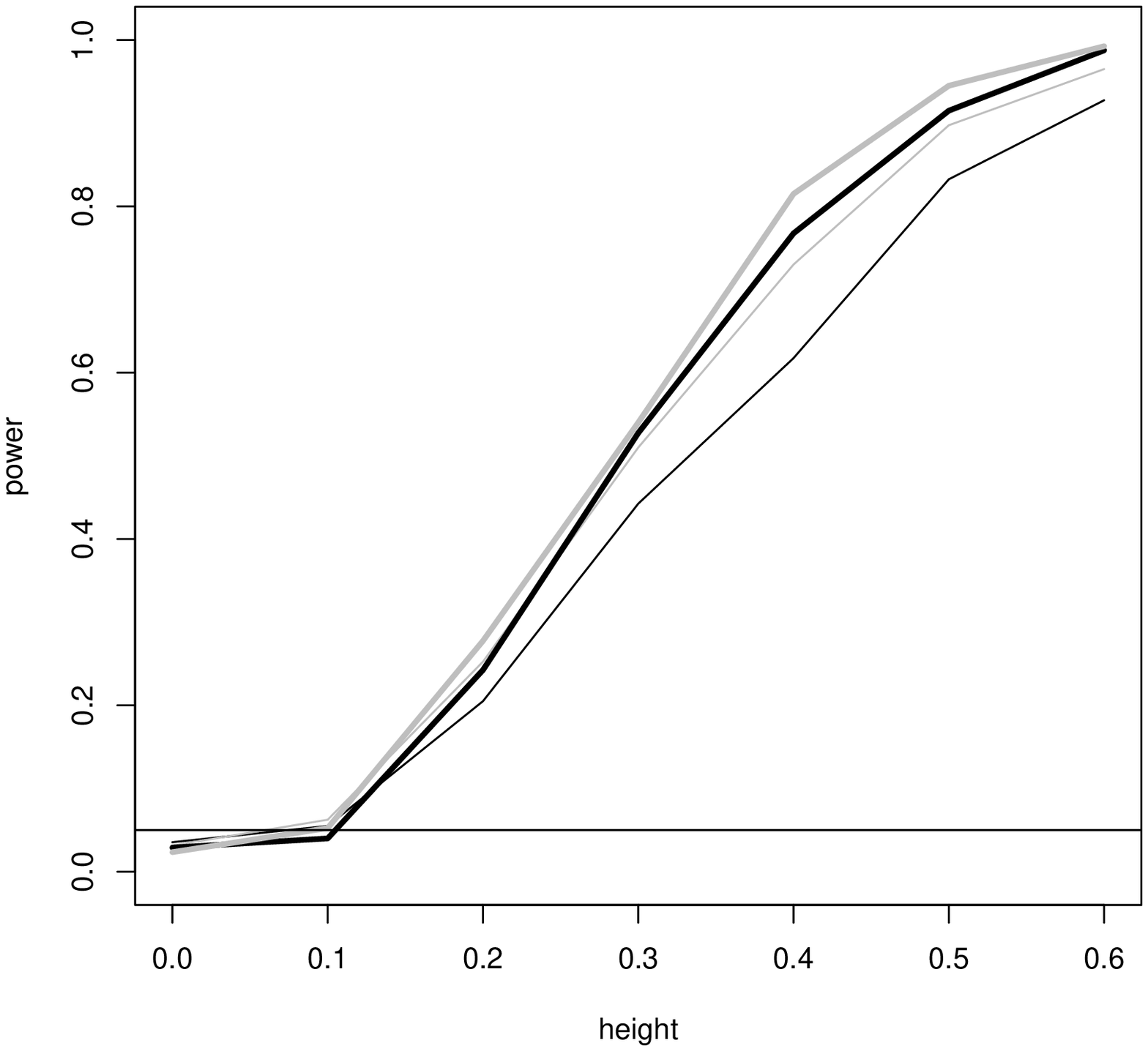}}\\
\resizebox{15cm}{7cm}{\includegraphics{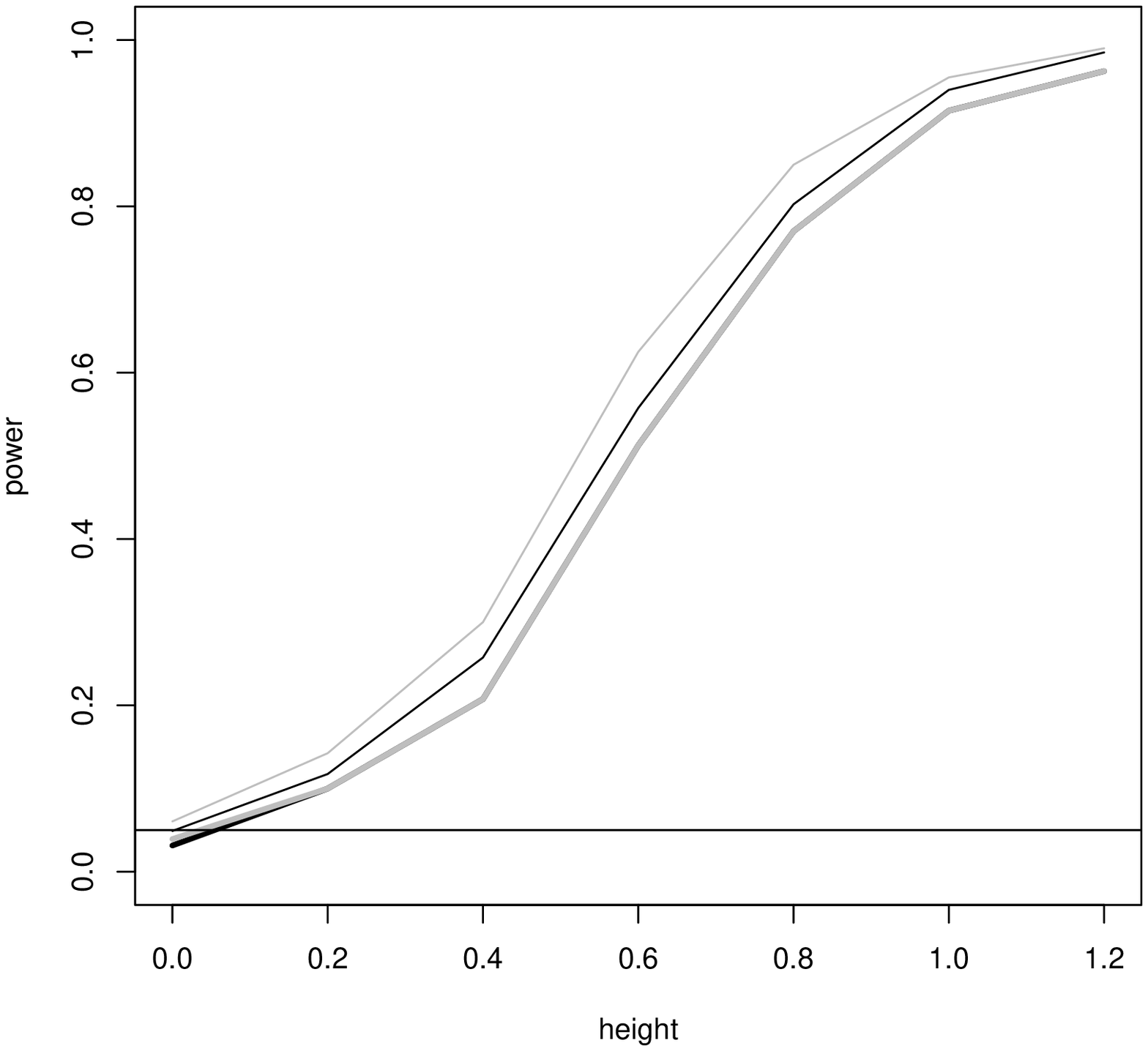}
\includegraphics{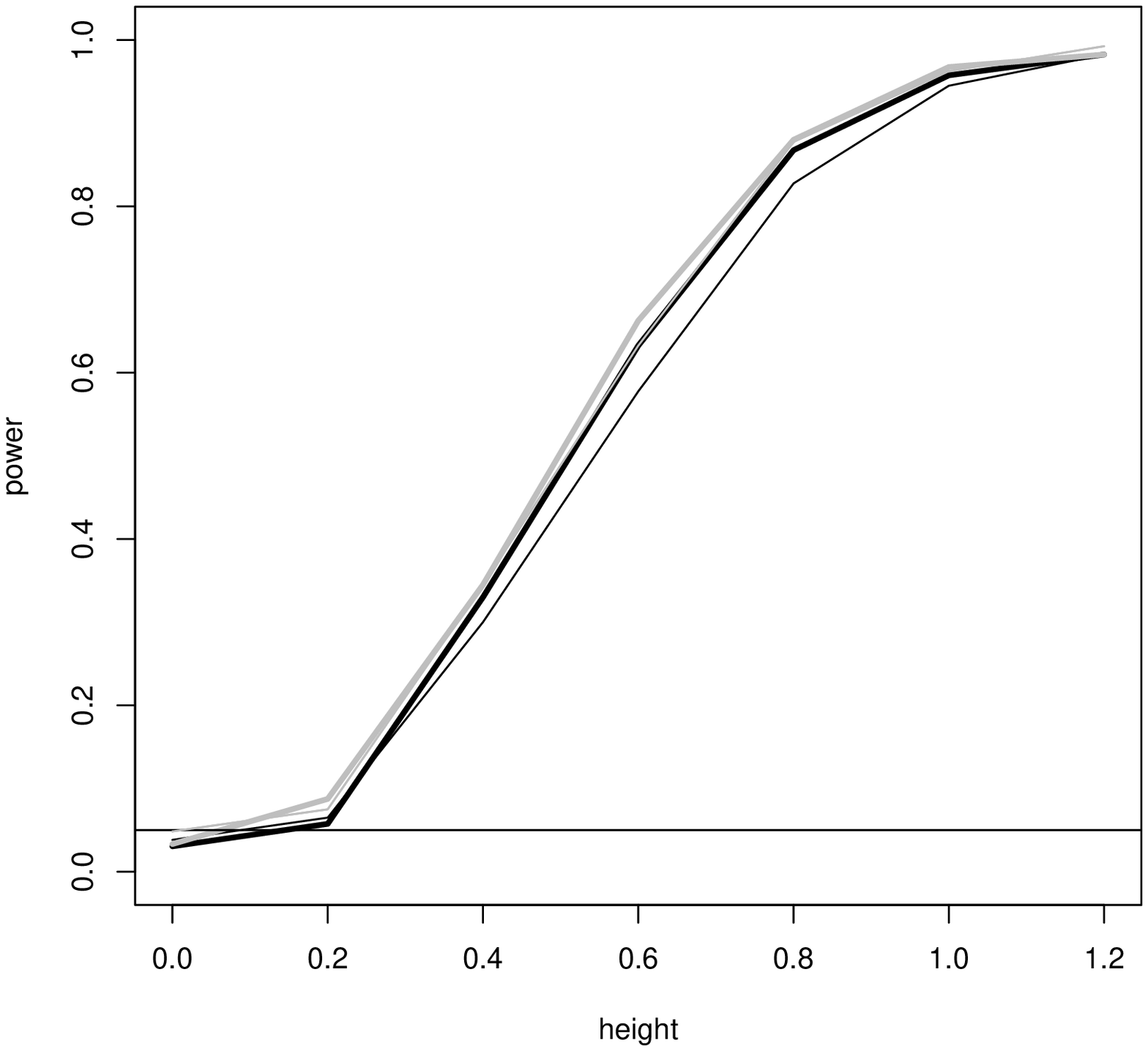}}\\
\resizebox{15cm}{7cm}{\includegraphics{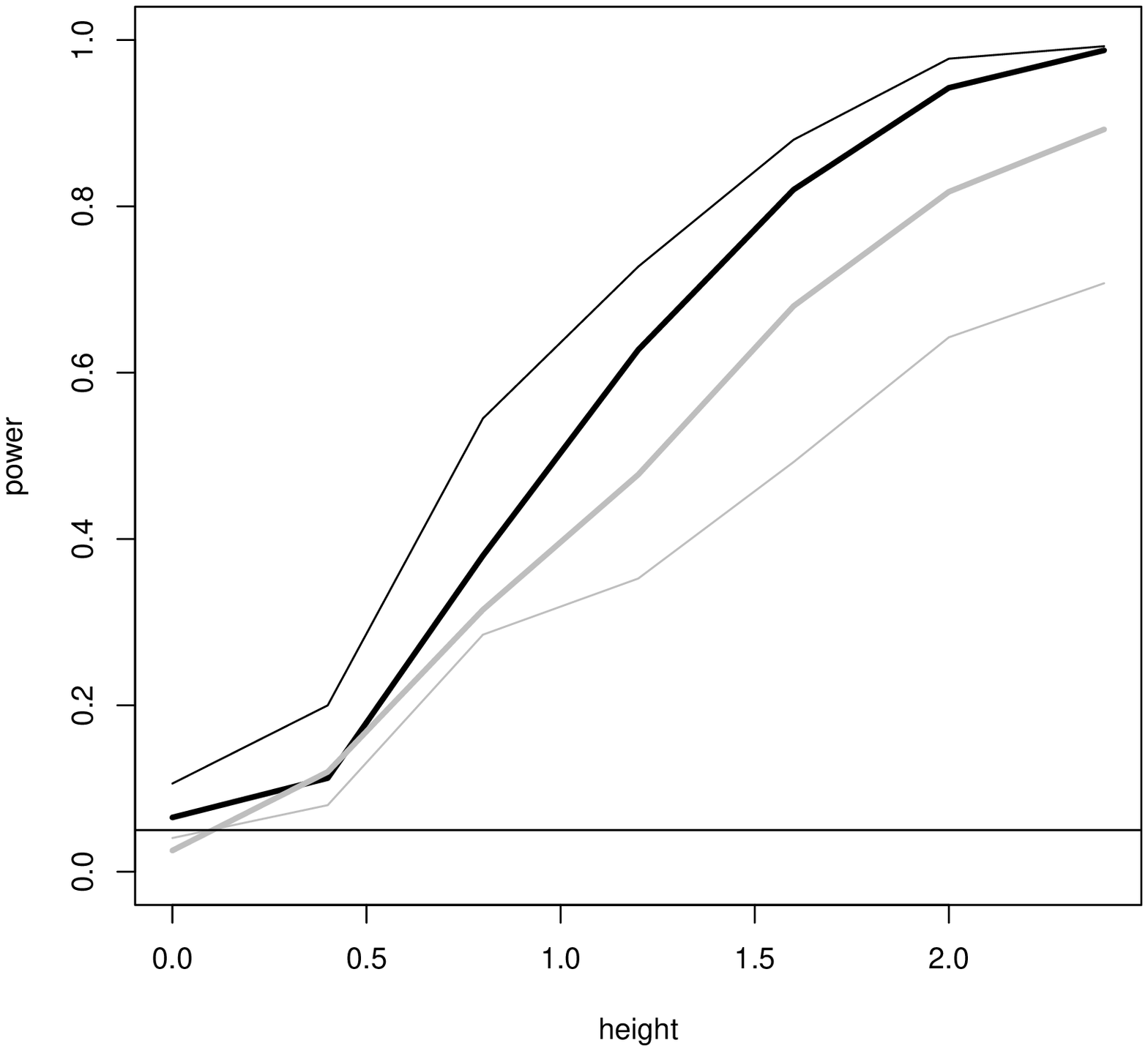}
\includegraphics{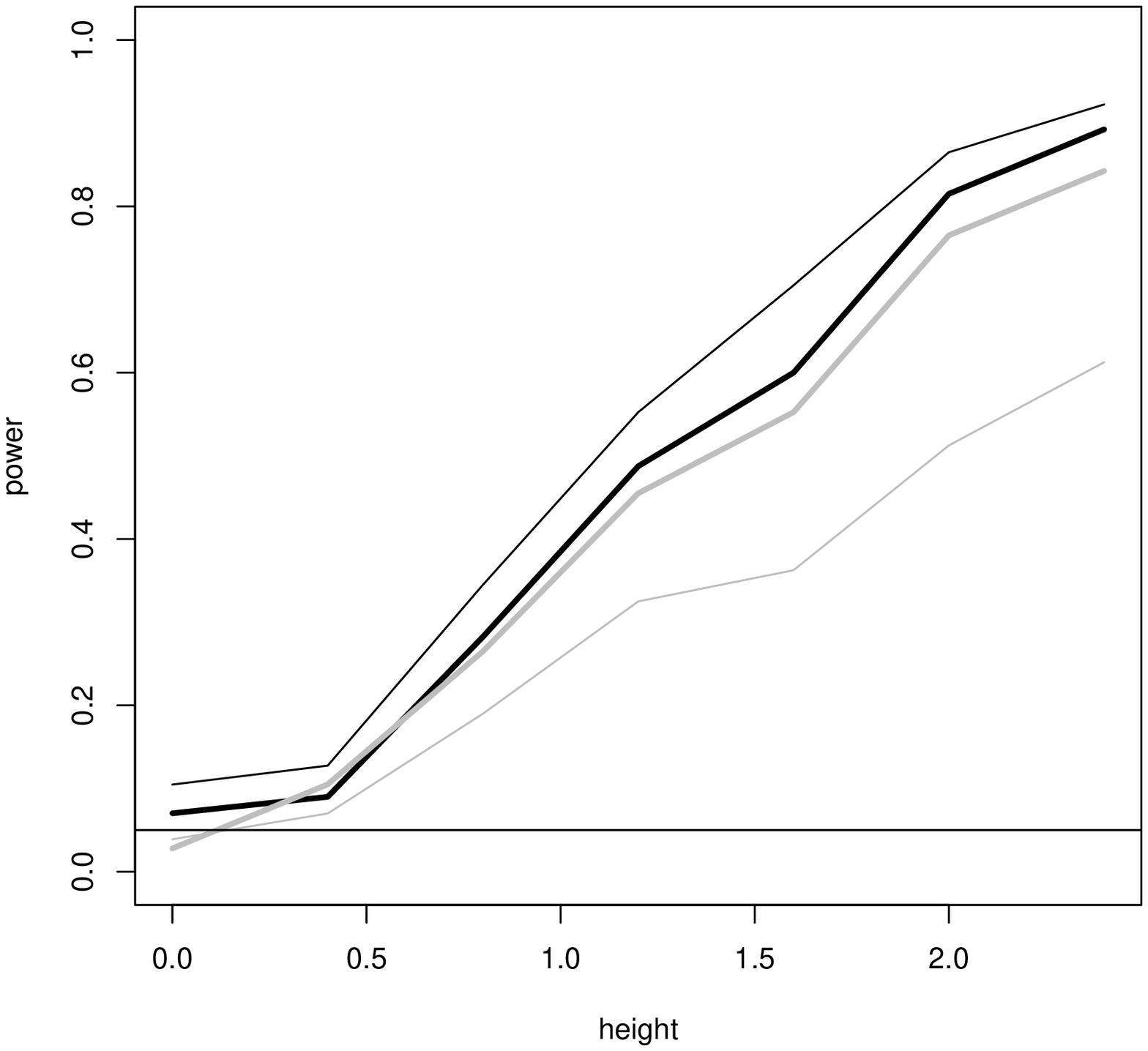}}\\
\end{figure}

The tests with Carlstein's adaptive choice of the block length could be improved further by using a more sophisticated estimate of $\rho$ than the ordinary sample autocorrelation used here. The latter is positively biased in the presence of a shift, which
 leads to too large choices of the block length. This negative effect becomes more severe for larger values of $\rho$, since the plug-in-estimate of the asymptotically MSE-optimal choice of $l_n$ increases more rapidly if $\hat{\rho}$ is close to 1, while it is rather stable for moderate and small values of $\hat{\rho}$. 
In our study, for $\rho=0$ the average value chosen for $l_n$ increases from about 2 to about 3, only, as the height of the shift increases, while it is from about 6 to about 9 if $\rho=0.4$, and even from about 16 to about 24 if $\rho=0.8$.
 An estimate of the autocorrelation coefficient which resists shifts could be used, 
 e.g. by applying a stepwise procedure which estimates the possible time of occurrence of a shift before calculating $\hat{\rho}$ from the corrected data, but this will not be pursued here.

\section{Auxiliary Results}
In this section, we will prove some auxiliary results which will play a crucial role in the proof of Theorem~\ref{LimitThm}. The main result of this section is the following proposition, which essentially shows that the degenerate part in the Hoeffding decomposition of the U-statistic $T_n(\lambda)$ is uniformly negligible.

\begin{proposition}
Let $(X_{n})_{n\geq 1}$ be a $1$-approximating functional with constants $(a_{k})_{k\geq 1}$ of an 
absolutely regular process with mixing coefficients $(\beta(k))_{k\geq 1}$, satisfying
\begin{equation}\label{summ2}
\sum^{\infty}_{k=1}k(\beta(k)+\sqrt{a_{k}}+\phi(a_{k}))<\infty. 
\end{equation}  
Moreover, let $g(x,y)$ be a $1$-continuous bounded degenerate kernel.  Then, as $n\rightarrow \infty$,
\begin{equation}
\frac{1}{n^{3/2}}  \sup_{0\leq \lambda\leq 1} \left|\sum_{i=1}^{[n\lambda]} \sum_{j=[n\lambda]+1}^n
 g(X_i,X_j) \right|\rightarrow 0
\label{eq:deg-conv}
\end{equation}
in probability.
\label{prop:deg-conv}
\end{proposition}

The proof of Proposition~\ref{prop:deg-conv} requires some moment bounds for increments of 
U-statistics of degenerate kernels, which we will now state as separate lemmas.

\begin{lemma}\label{second}
Let $(X_{n})_{n\geq 1}$ be a $1$-approximating functional with constants $(a_{k})_{k\geq 1}$ of an 
absolutely regular process with mixing coefficients $(\beta(k))_{k\geq 1}$, satisfying
\begin{equation}
\sum^{\infty}_{k=1}k(\beta(k)+\sqrt{a_{k}}+\phi(a_{k}))<\infty. 
\end{equation}  
Moreover, let $g(x,y)$ be a $1$-continuous bounded degenerate kernel.  
Then, there exists a constant $C_{1}$ such that
\begin{equation}
E\left(\sum^{[n\lambda]}_{i=1}\sum^{n}_{j=[n\lambda]+1}g(X_{i},X_{j})\right)^{2}\leq C_{1}[n\lambda](n-[n\lambda]).
\end{equation}
\end{lemma}

\begin{proof}
We can write
\begin{multline}
E\left(\sum^{[n\lambda]}_{i=1}\sum^{n}_{j=[n\lambda]+1}g(X_{i},X_{j})\right)^{2}
=\sum^{[n\lambda]}_{i=1}\sum^{n}_{j=[n\lambda]+1}E(g(X_{i},X_{j}))^{2} \\ +
2\sum_{1\leq i_{1}\neq i_{2}\leq [n\lambda]}\sum_{[n\lambda]+1\leq j_{1}\neq j_{2}\leq n}E\left(g(X_{i_{1}},X_{j_{1}})g(X_{i_{2}},X_{j_{2}})\right)
\end{multline}
The elements of the first sum all are bounded, hence
\begin{equation}
\sum^{[n\lambda]}_{i=1}\sum^{n}_{j=[n\lambda]+1}E(g(X_{i},X_{j}))^{2}\leq C[n\lambda](n-[n\lambda]).
\end{equation}
Concerning the second sum, by Lemma $\ref{first}$, we get 
\begin{eqnarray}
&&\hspace{-25mm} \sum_{1\leq i_{1}< i_{2}\leq [n\lambda]}\sum_{[n\lambda]+1\leq j_{1}< j_{2}\leq n}E\left(g(X_{i_{1}},X_{j_{1}})g(X_{i_{2}},X_{j_{2}})\right) \nonumber \\
&& \qquad \leq 4\, S\sum_{1\leq i_{1}< i_{2}\leq [n\lambda]}\sum_{[n\lambda]+1\leq j_{1}\leq j_{2}\leq n}\phi(a_{[k/3]}) \nonumber \\
&& \qquad \qquad+ 8\, S^{2}\sum_{1\leq i_{1}< i_{2}\leq [n\lambda]}\sum_{[n\lambda]+1\leq j_{1}\leq j_{2}\leq n}(\sqrt{a_{[k/3]}}+\beta([k/3]))
\end{eqnarray}
with $k=\max\{|i_2-i_1|,|j_2-j_1|\}$. We will first treat the summands with $k=i_{2}-i_{1}$. Suppose for one moment that $k$ is fixed and we will bound the number of indices that appear in the sum. Observe that in this case we have $[n\lambda]$ ways to choose $i_{1}$, once $i_{1}$ is chosen we have one way to pick $i_{2}$ because $i_{2}=i_{1}+k$. For $j_{1}$ we have as before $n-[n\lambda]$ ways to pick this index and then for each $j_{1}$, $j_{2}$ need to be in the interval $[j_{1},j_{1}+k]$ and there are exactly $k$ integers in such interval.
\begin{multline}
\sum_{1\leq i_{1}< i_{2}\leq [n\lambda]}\sum_{[n\lambda]+1\leq j_{1}< j_{2}\leq n}\left(4S\phi(a_{[k/3]}+8S^2\sqrt{a_{[k/3]}}+8S^2\beta([k/3]))\right)\\
\leq C[n\lambda](n-[n\lambda])\left(\sum^{n}_{k=1}k\phi(a_{k})+\sum^{n}_{k=1}k\sqrt{a_{k}}+\sum^{n}_{k=1}k\beta(k)\right)\leq C[n\lambda](n-[n\lambda])
\end{multline}
Analogously we can find the bounds for the terms with $k=i_{1}-i_{2}$, $k=j_{2}-j_{1}$ and $k=j_{1}-j_{2}$  using the conditions of summability.   
\end{proof}

We now define the process $G(\lambda)$, $0\leq \lambda \leq 1$, by
\begin{equation}
G_{n}(\lambda):=n^{-3/2}\sum^{[n\lambda]}_{i=1}\sum^{n}_{j=[n\lambda]+1}g(X_{i},X_{j}),\quad 0\leq \lambda \leq 1.
\end{equation}

\begin{lemma}\label{third} Under the conditions of Lemma \ref{second}, there exists a constant $C$ such that
\begin{equation}
E(|G_{n}(\eta)-G_{n}(\mu)|^2)\leq\frac{C}{n}(\eta-\mu),
\label{eq:g-bound}
\end{equation}
for all $0\leq \mu \leq \eta\leq 1$.
\end{lemma}

\begin{proof}
We can write
\begin{eqnarray}
\label{ineq} 
&& \hspace{-10mm} E(|G_{n}(\eta)-G_{n}(\mu)|^2) \\
&\leq& \frac{2}{n^3}E\left(\sum^{[n\mu]}_{i=1}\sum^{[n\eta]}_{j=[n\mu]+1}g (X_{i},X_{j})\right)^2 + \frac{2}{n^3}E\left(\sum^{[n\eta]}_{i=[n\mu]+1}\sum^{n}_{j=[n\eta]+1}g (X_{i},X_{j}) \right)^2 \nonumber\\
&=& \frac{2}{n^3}E\left(\sum^{[n\mu]}_{i=1}\sum^{[n\eta]}_{j=[n\mu]+1}g (X_{i},X_{j})\right)^2 +\frac{2}{n^3}E\left(\sum^{[n\eta]-[n\mu]}_{i=1}\sum^{n-[n\mu]}_{j=[n\eta]-[n\mu]+1}g (X_{i},X_{j}) \right)^2
\nonumber\\
&\leq& C\frac{1}{n^3}\left([n\mu]([n\eta]-[n\mu])+([n\eta]-[n\mu])(n-[n\eta])\right)\leq\frac{C}{n}(\eta-\mu)
\nonumber
\end{eqnarray}
using the stationarity of the process $(X_n)_{n\in \N}$ and Lemma $\ref{second}$.
\end{proof}
\vspace{3mm}
{\em Proof of Proposition~\ref{prop:deg-conv}.}
From Lemma~\ref{third} we obtain, using Chebyshev's inequality,
\begin{equation}
P\left(|G_{n}(\eta)-G_{n}(\mu)|\geq \epsilon\right)\leq \frac{1}{\epsilon^2}\frac{C}{n}(\eta-\mu),
\end{equation}
for all $\epsilon>0$. Thus we get  for $0\leq k\leq m\leq n $ with $k,m,n \in \N$
\begin{eqnarray}
P\left(\left|G_{n}\left(\frac{m}{n}\right)-G_{n}\left(\frac{k}{n}\right)\right|\geq \epsilon\right)
&\leq & \frac{1}{\epsilon^2}E\left(G_{n}\left(\frac{m}{n}\right)-G_{n}\left(\frac{k}{n}\right)\right)^2
\nonumber \\
&\leq & \frac{1}{\epsilon^2}\frac{C}{n^2}(m-k)\leq \frac{1}{\epsilon^2}\frac{C}{n^{5/3}}(m-k)^{4/3} 
\label{ineq1}
\end{eqnarray}
as $m-k\leq n$. Now consider the variables 
\begin{equation}
\label{vari}
\zeta_{i}=\left\{
\begin{array}{cl}
G_{n}\left(\frac{i}{n}\right)-G_{n}\left(\frac{i-1}{n}\right)&\mbox{if } i=1,\ldots,n-1\\
0&\mbox{else }
\end{array}\right.
\end{equation}
and suppose that $S_{i}=\zeta_{1}+\zeta_{2}+\ldots+\zeta_{i}$ with $S_{0}=0$, then $S_{i}=G_{n}(\frac{i}{n})$. In consequence the inequality $(\ref{ineq1})$ is equivalent to
\begin{equation}
P(|S_{m}-S_{k}|\geq \epsilon)\leq \frac{1}{\epsilon^2}\left[\frac{C^{3/4}}{n^{5/4}}(m-k)\right]^{4/3}\quad \text{for}\quad 0\leq k\leq m\leq n.
\label{ineq2}
\end{equation}
So the assumption of Theorem \ref{Billing} are satisfaced with the variables $(\ref{vari})$ in the role of the $\xi_{i}$, $\beta=1/2$, $\alpha=2/3$ and $u_{l}=C^{3/4}/n^{5/4}$, $u_{o}=0$ and hence
\begin{equation}
P\left(\max_{1\leq i\leq n-1}|S_{i}|\geq \epsilon\right)
\leq \frac{K}{\epsilon^2}\left[\frac{C^{3/4}}{n^{5/4}}(n-1)\right]^{4/3}\leq \frac{KC}{\epsilon^2n^{1/3}}
\end{equation}
where $K$ depends only of $\alpha$ and $\beta$. Thus, $(\ref{eq:deg-conv})$ holds as $n\rightarrow\infty$.
\hfill $\Box$

\section{Proof of Main Results}

In this section, we will prove Theorem~\ref{LimitThm} and Theorem~\ref{th:anti-symm}. Note that Theorem~\ref{th:anti-symm} is a direct consequence of Theorem~\ref{LimitThm}, applied to anti-symmetric kernels. We will nevertheless present a direct proof of Theorem~\ref{th:anti-symm}, since this proof is much simpler than the proof in the general case. Moreover, Theorem~\ref{th:anti-symm} covers those cases that are most relevant in applications.

The first part of the proof is identical for both Theorem~\ref{LimitThm} and Theorem~\ref{th:anti-symm}.
Note that, for each $\lambda\in [0,1]$, the statistic $T_n(\lambda)$ is a two-sample U-statistic. Thus,
using the Hoeffding decomposition (\ref{eq:h-decomp}), we can write $T_{n}(\lambda)$ as
\begin{align}
T_{n}(\lambda)&=\frac{1}{n^{3/2}}\left(\sum^{[\lambda n]}_{i=1}\sum^{n}_{j=[\lambda n]+1}(h_{1}(X_{i})+h_{2}(X_{j})+g(X_{i},X_{j}))\right) \nonumber\\
& =\frac{1}{n^{3/2}}\left((n-[n\lambda])\sum^{[n\lambda]}_{i=1}h_{1}(X_{i})+[n\lambda]\sum^{n}_{j=
[n\lambda]+1}h_{2}(X_{j})+\sum^{[\lambda n]}_{i=1}\sum^{n}_{j=[\lambda n]+1}g(X_{i},X_{j})\right)
\end{align}
By Proposition~\ref{prop:deg-conv}, we know that
\[
  \frac{1}{n^{3/2}} \sup_{0\leq \lambda \leq 1} \left|\sum^{[\lambda n]}_{i=1}\sum^{n}_{j=[\lambda n]+1}g(X_{i},X_{j}) \right| \rightarrow 0
\]
in probability. Thus, by Slutsky's lemma, it suffices to show that the sum of the first two terms, i.e.
\begin{equation}
  \left( \frac{n-[n\lambda]}{n^{3/2}} \sum^{[n\lambda]}_{i=1}h_{1}(X_{i})+\frac{[n\lambda]}{n^{3/2}} \sum^{n}_{j=
[n\lambda]+1}h_{2}(X_{j})\right)_{0\leq \lambda\leq 1}
\label{eq:main-term}
\end{equation}
converges in distribution to the desired limit process.

\begin{proof}[Proof of Theorem~\ref{th:anti-symm}] It remains to show that (\ref{eq:main-term}) converges in distribution to $\sigma W^{(0)}(\lambda), 0\leq \lambda\leq 1$, where $(W^{(0)}(\lambda))_{0\leq \lambda \leq 1}$ is standard Brownian bridge on $[0,1]$, and where $\sigma^2$ is defined in (\ref{eq:anti-symm-var}).
By antisymmetry of the kernel $h(x,y)$, we obtain that $h_2(x)=-h_1(x)$. Hence, in this case, (\ref{eq:main-term}) can be rewritten as
\[
\frac{n-[n\lambda]}{n^{3/2}} \sum_{i=1}^{[n\lambda]} h_1(X_i) -\frac{[n\lambda]}{n^{3/2}} \sum_{i=[n\lambda]+1}^n h_1(X_i)  = \frac{1}{n^{1/2}} \sum_{i=1}^{[n\lambda]} h_1(X_i)
 -\frac{[n\lambda]}{n^{3/2}} \sum_{i=1}^n h_1(X_i).
\]
By Proposition~2.11 and Lemma~2.15 of Borovkova, Burton and Dehling (2001), the sequence $(h_1(X_i))_{i\geq1}$ is a 1-approximating functional with approximating constant $C\sqrt{a_k}$. Since $h_1(X_i)$ is bounded, the $L_2$-near epoch dependence in the sense of Wooldridge and White (1988) also holds, with the same constants. Moreover, the underlying process $(Z_n)_{n\geq 1}$ is absolutely regular, and hence also strongly mixing. Thus we may apply the invariance principle in Corollary~3.2 of Wooldridge and White (1988), and obtain that the partial sum process 
\begin{equation}
\left(\frac{1}{n^{1/2}}\sum^{[n\lambda]}_{i=1}h_{1}(X_{i})\right)_{0\leq \lambda\leq 1}
\end{equation}
converges weakly to Brownian motion $(W(\lambda))_{0\leq \lambda \leq 1}$ with $\Var(W(1))=\sigma^2$. The statement of the Theorem follows with the continuous mapping theorem for the mapping
$ x(t) \mapsto x(t)-tx(1),\; 0\leq t\leq 1$.
\end{proof}

The proof of Theorem~\ref{LimitThm} requires an invariance principle for the partial sum process of 
$\R^2$-valued dependent random variables; see Proposition~\ref{invprin} below. For mixing processes, such invariance principles have been established even for partial sums of Hilbert space valued random vector, e.g. by Dehling (1983). In this paper, we provide an extension of these results to functionals of mixing processes.

\begin{proposition}\label{invprin}
Let $(X_n)_{n\in\N}$ be a $1-$approximating functional of an absolutely re\-gular process with mixing coefficients $(\beta(k))$ and let $h_{1}(\cdot)$, $h_{2}(\cdot)$ be bounded $1-$Lipschitz functions with mean zero.  Suppose that the sequences $(\beta(k))_{k\geq0}$, $(a_{k})_{k\geq0}$ and $(\phi(a_{k}))_{k\geq0}$ satisfy 

\begin{equation}\label{summ}
\sum_{k}k^{2}(\beta(k)+a_{k}+\phi(a_{k}))<\infty.
\end{equation} 
Then, as $n\rightarrow\infty$,
 
\begin{equation}\label{limit}
\left(\frac{1}{\sqrt{n}}\sum^{[nt]}_{i=1}\begin{pmatrix} h_{1}(X_{i}) \\ h_{2}(X_{i})\end{pmatrix}\right)_{0\leq t \leq 1} \longrightarrow \begin{pmatrix} W_{1}(t) \\ W_{2}(t)\end{pmatrix}_{0\leq t\leq 1} 
\end{equation}
where $(W_1(t),W_2(t))_{0\leq t\leq 1}$ is a two-dimensional Brownian motion with mean zero and covariance
$E(W_k(s)\, W_l(t))=\min(s,t) \sigma_{kl}$,
for $0\leq s,t\leq 1$ with $\sigma_{k,l}$ as defined in $(\ref{variance})$.
\end{proposition}

\begin{proof}
To prove $(\ref{limit})$, we need to establish finite dimensional convergence and tightness. 
Concerning finite-dimensional convergence,  by the Cram\'er-Wold device it suffices to show the convergence in distribution of a linear combination of the coordinates of the vector
\begin{multline}
  \left(\frac{1}{\sqrt{n}} \sum^{[nt_{1}]}_{i=1}h_{1}(X_{i}),\frac{1}{\sqrt{n}} \sum^{[nt_{1}]}_{i=1}h_{2}(X_{i}), \ldots,  \frac{1}{\sqrt{n}} \sum^{[nt_{j}]}_{i=1}h_{1}(X_{i}),\frac{1}{\sqrt{n}} \sum^{[nt_{j}]}_{i=1}h_{2}(X_{i})),\right.\\
   \left.\ldots, \frac{1}{\sqrt{n}} \sum^{n}_{i=1}h_{1}(X_{i}),\frac{1}{\sqrt{n}}\sum^{n}_{i=1}h_{2}(X_{i})\right),
\end{multline}
for $0=t_{0}<t_{1}<\ldots<t_{j}<\ldots<t_{k}=1$.  Any such linear combination can be expressed as
\begin{equation}
\sum_{j=1}^k\frac{1}{\sqrt{n}} \sum^{[nt_{j}]}_{i=[nt_{j-1}]+1}(a_{j}h_{1}(X_{i})+b_{j}h_{2}(X_{i})),
\end{equation}
for $(a_{j},b_{j})^{k}_{j=1}\in\R^{2\, k}$. By using the Cram\'er-Wold device again,  the weak convergence of this sum is equivalent to the weak convergence of the vector
\begin{multline}\label{vector}
  \left(\frac{1}{\sqrt{n}} \sum^{[nt_{1}]}_{i=1}(a_{1}h_{1}(X_{i})+b_{1}h_{2}(X_{i})), \ldots,  \frac{1}{\sqrt{n}} \sum^{[nt_{j}]}_{i=[nt_{j-1}]+1}(a_{j}h_{1}(X_{i})+b_{j}h_{2}(X_{i})),\right.\\
   \left.\ldots, \frac{1}{\sqrt{n}} \sum^{n}_{i=[nt_{k-1}]+1}(a_{k}h_{1}(X_{i})+b_{k}h_{2}(X_{i}))\right)
\end{multline}
to
\begin{multline}
\big(a_1(W_1(t_{1})-W_1(t_{0}))+b_1(W_2(t_{1})-W_2(t_{0})),\ldots,\\
a_k(W_1(t_{k})-W_1(t_{k-1}))+b_k(W_2(t_{k})-W_2(t_{k-1}))\big).
\end{multline}
Since $(X_{n})_{n\geq1}$ is a $1-$approximating functional, it can be coupled with a process consisting of independent blocks.  Given integers $L:=L_{n}=[n^{3/4}]$ and $l_{n}=[n^{1/2}]$, we introduce the $(l,L)$ blocking $(B_{m})_{m\geq0}$ of the variables $(a_{j}h_{1}(X_{i})+b_{j}h_{2}(X_{i}))$ with $i=[nt_{j-1}]+1,\ldots, [nt_{j}]$, $j=0,\ldots,k$ and
\begin{equation}
B_{m}:=\sum^{m(L_{n}+(m-1)l_{n})}_{i=(m-1)(L_{n}+l_{n})+1}(a_{j}h_{1}(X_{i})+b_{j}h_{2}(X_{i}))
\end{equation}
and separating blocks
\begin{equation}
\tilde{B}_{m}:=\sum^{m(L_{n}+l_{n})}_{i=mL_{n}+(m-1)l_{n}+1}(a_{j}h_{1}(X_{i})+b_{j}h_{2}(X_{i})).
\end{equation}
By Theorem \ref{The3} there exists a sequence of independent blocks $(B'_{m})$ with the same blockwise marginal distribution as $(B_{m})$ and such that 
\[
P\left(|B_{m}-B'_{m}|\leq 2\alpha_{l}\right)\geq 1-\beta(l)-2\alpha_{l},
\]
where $\alpha_{l}:=\left(2\sum^{\infty}_{k=[l_{n}/3]}a_{k}\right)^{1/2}$.
We can express the components of our vector (\ref{vector}) as a sum of blocks
\begin{multline}
\sum^{[nt_{j+1}]}_{i=[nt_{j}]+1}(a_{j}h_{1}(X_{i})+b_{j}h_{2}(X_{i}))\\
= \sum_{m=\left[\frac{nt_j}{L+l}\right]+1}^{\left[\frac{nt_{j+1}}{L+l}\right]}B_m+\sum_{m=\left[\frac{nt_j}{L+l}\right]+1}^{\left[\frac{nt_{j+1}}{L+l}\right]}\tilde{B}_m+\sum_{R_j}(a_{j}h_{1}(X_{i})+b_{j}h_{2}(X_{i})),
\end{multline}
where $R_j$ denotes the set of indices not contained in the blocks. Observe that by the Lemma \ref{lem2.23} for any set $A\subset \{1,\ldots,n\}$
\begin{equation}
E\left(\sum_{i\in A}(a_{j}h_{1}(X_{i})+b_{j}h_{2}(X_{i}))\right)^{2}\leq C \#A
\end{equation}
and hence
\begin{equation}
E\left(\sum_{m=\left[\frac{nt_j}{L+l}\right]+1}^{\left[\frac{nt_{j+1}}{L+l}\right]}\tilde{B}_m\right)^2\leq C\frac{n}{L_n+l_n}l_n\leq Cn^{3/4},
\end{equation}
so it follows with the Chebyshev inequality that this term is negligible. For the last summand, we have that
\begin{equation}
E\left(\sum_{R_j}(a_{j}h_{1}(X_{i})+b_{j}h_{2}(X_{i}))\right)^2\leq C2(L_n+l_n)\leq Cn^{3/4}.
\end{equation}
Furthermore, we need to show that we can replace the blocks $B_m$ by the independent coupled blocks $B'_m$:
\begin{eqnarray*}
P\left(\left|\frac{1}{\sqrt{n}}\sum_{m=\left[\frac{nt_j}{L+l}\right]+1}^{\left[\frac{nt_{j+1}}{L+l}\right]}(B_m-B'_m)\right|>\epsilon\right)
&\leq& \sum_{m=\left[\frac{nt_j}{L+l}\right]+1}^{\left[\frac{nt_{j+1}}{L+l}\right]}P\left(|B_m-B'_m|>\frac{\epsilon\sqrt{n}}{n^{1/4}}\right) \\
&\leq& n^{\frac{1}{4}}\left(\beta([\frac{l_n}{3}])+\alpha_{[\frac{l_n}{3}]}\right)\rightarrow0
\end{eqnarray*}
as $n\rightarrow\infty$ by our conditions on the mixing coefficients and approximation constants. Here we used that fact that $\alpha_n\rightarrow0$ and thus, for almost all $n\in\N$,
\begin{equation}
P\left(|B_m-B'_m|>\epsilon n^{1/4}\right)\leq P\left(|B_m-B'_m|>2\alpha_{l_n}\right).
\end{equation}
With the above arguments the result holds if we show the convergence of
\begin{equation}
\frac{1}{\sqrt{n}}\left(\sum_{m=\left[\frac{nt_0}{L+l}\right]+1}^{\left[\frac{nt_{1}}{L+l}\right]}B'_{m},\ldots,\sum_{m=\left[\frac{nt_k}{L+l}\right]+1}^{\left[\frac{nt_{k+1}}{L+l}\right]}B'_{m}\right).
\end{equation}
Since this vector has independent components, we only need to show the one-dimensional convergence, which is a consequence of Theorem \ref{The4}, using the summability condition $(\ref{summ})$.    

We now turn to the question of tightness and show that, for each $\epsilon$ and $\eta$, there exist a $\delta$, $0<\delta<1$, and an integer $n_{0}$ such that, for $0\leq t\leq1$,
\begin{equation}
\frac{1}{\delta}P\left(\sup_{t\leq s\leq t+\delta}|Y_{n}(s)-Y_{t}|\geq \epsilon\right)\leq\eta,\quad n\geq n_{0}
\end{equation}
with
\begin{equation}
Y_{n}(t)=\frac{1}{\sigma\sqrt{n}}\sum_{i=1}^{[nt]}h_1(X_i)+(nt-[nt])\frac{1}{\sigma\sqrt{n}}X_{[nt]+1}
\end{equation}
($h_2$ can be treated in the same way) and by Theorem $\ref{Billing1}$, this condition reduces to: For each positive $\epsilon$ there exist a $\alpha >1$ and an integer $n_{0}$, s. t.
\begin{equation}
P\left(\max_{i\leq n}\left|\sum^{i}_{j=1}h_{1}(X_{j})\right| \geq \lambda\sqrt{n}\right)\leq \frac{\epsilon}{\lambda^{2}},\quad n\geq n_{0}.
\end{equation}
Let $t\geq s$, $s,t\in [0,1]$. By Lemma $\ref{lem2.24}$ we get 

\begin{eqnarray}
E\left(\left|\frac{1}{\sqrt{n}}\sum^{[nt]}_{i=1}h_{1}(X_{i})-\frac{1}{\sqrt{n}}\sum^{[ns]}_{i=1}h_{1}(X_{i})\right|^{4}\right)&=& \frac{1}{n^{2}}E\left(\sum^{[nt]}_{[ns]+1}h_{1}(X_{i})\right)^{4}\nonumber\\
& \leq &\frac{1}{n^2}(([nt]-[ns])^{2}C)
\end{eqnarray}
and this implies
\begin{equation}
P \left(\left|\frac{1}{\sqrt{n}}\sum^{m}_{i=1}h_{1}(X_{i})-\frac{1}{\sqrt{n}}\sum^{k}_{i=1}h_{1}(X_{i})\right|\geq \epsilon\right)\leq \frac{1}{\epsilon^{4}}\left(\frac{C^{1/2}}{n}(m-k)\right)^{2}.
\end{equation}
By Theorem \ref{Billing}
\begin{equation}
P\left(\max_{i\leq n}\left|\sum^{i}_{j=1}h_{1}(X_{j})\right|\geq \epsilon\sqrt{n}\right)\leq \frac{K}{\epsilon^{4}}\left(\frac{C^{1/2}}{n}(n-1)\right)^{2}
\end{equation}
and we get the assertion.    In this way, we have established tightness of each of the two coordinates of the 
partial sum process. This also implies tightness of the vector-valued process.
\end{proof}

\begin{proof}[Proof of Theorem~\ref{LimitThm}]
From Proposition $\ref{invprin}$ we obtain that
\begin{equation}
\left(\frac{1}{\sqrt{n}}\sum^{[n\lambda]}_{i=1}\begin{pmatrix} h_{1}(X_{i})\\h_{2}(X_{i})\end{pmatrix}\right)_{0\leq \lambda\leq 1}\longrightarrow 
\begin{pmatrix} W_1(\lambda) \\ W_2(\lambda) \end{pmatrix}_{0\leq \lambda\leq 1},
\label{eq:2D-ip}
\end{equation}
in distribution on the space $(D([0,1]))^2$. 
We consider the functional given by
\begin{equation}
\begin{pmatrix} x_{1}(t) \\ x_{2}(t) \end{pmatrix}\mapsto (1-t)x_{1}(t)+t(x_{2}(1)-x_{2}(t)), \quad 0\leq t\leq 1.
\end{equation}
This is a continuous mapping from $(D[0,1])^2$ to $D[0,1]$, so we may apply the continuous mapping theorem to (\ref{eq:2D-ip}), and obtain
\begin{multline*}
 \left(\frac{n-[n\lambda]}{n^{3/2}}\sum^{[n\lambda]}_{i=1}h_{1}(X_{i})
+\frac{[n\lambda]}{n^{3/2}} \sum^{n}_{j=
[n\lambda]+1}h_{2}(X_{j})\right)_{0\leq \lambda\leq 1}\\
\longrightarrow\left((1-\lambda)W_{1}(\lambda)+\lambda (W_{2}(1)-W_{2}(\lambda))\right)_{0\leq \lambda\leq 1}.
\end{multline*}
Together with the remarks at the beginning of this section, this proves Theorem~\ref{LimitThm}.
\end{proof}

\section{Appendix: Some Auxiliary Results from the Literature}

In this section, we collect some known lemmas and theorems for weakly dependent data. We start with some results on the behaviour of partials sums:

\begin{lemma}[Lemma 2.23 \cite{BBD}]\label{lem2.23}
Let $(X_{k})_{k\in \Z}$ be a $1-$approximating functional with constants $(a_{k})_{k\geq0}$ of an absolutely regular process with mixing coefficients $(\beta(k))_{k\geq 0}$. Suppose moreover that $E X_{i}=0$ and that one of the following two conditions holds:
\begin{enumerate}
\item $X_{0}$ is bounded a.s. and $\sum^{\infty}_{k=0}(a_{k}+\beta(k))<\infty.$
\item $E|X_{0}|^{2+\delta}<\infty$ and $\sum^{\infty}_{k=0}(a^{\frac{\delta}{1+\delta}}_{k}+\beta^{\frac{\delta}{1+\delta}}(k))<\infty.$
\end{enumerate}
Then, as $N\rightarrow\infty$, 
\begin{equation}
\frac{1}{N}E S^{2}_{N}\rightarrow E X^{2}_{0}+2\sum^{\infty}_{j=1}E(X_{0}X_{j})
\end{equation}
and the sum on the r.h.s. converges absolutely.
\end{lemma}

\begin{lemma}[Lemma 2.24 \cite{BBD}]\label{lem2.24}
Let $(X_{k})_{k\in \Z}$ be a $1-$approximating functional with constants $(a_{k})$ of an absolutely regular process with mixing coefficients $(\beta(k))_{k\geq 0}$.\;  Suppose moreover that $E X_{i}=0$ and that one of the following two conditions holds:
\begin{enumerate}
\item $X_{0}$ is bounded a.s. and $\sum^{\infty}_{k=0}k^{2}(a_{k}+\beta(k))<\infty.$
\item $E|X_{0}|^{4+\delta}<\infty$ and $\sum^{\infty}_{k=0}k^{2}(a^{\frac{\delta}{3+\delta}}_{k}+\beta^{\frac{\delta}{4+\delta}}(k))<\infty.$
\end{enumerate}
Then there exits a constant $C$ such that
\begin{equation}
E S^{4}_{N}\leq CN^{2}.
\end{equation}
\end{lemma}

\begin{theorem}[Theorem 4 \cite{BBD}]\label{The4} 
Let $(X_{k})_{k\in\Z}$ be a $1-$approximating functional with constants $(a_{k})_{k\geq 0}$ of an absolutely regular process with mixing coefficients $(\beta(k))_{k\geq 0}$. Suppose moreover that $E X_{i}=0$, $E|X_{0}|^{4+\delta}<\infty$ and that  
\begin{equation}\label{summa}
\sum^{\infty}_{k=0}k^{2}(a^{\frac{\delta}{3+\delta}}_{k}+\beta^{\frac{\delta}{4+\delta}}(k))<\infty,
\end{equation}
for some $\delta>0$. Then, as $n\rightarrow \infty,$ 
\begin{equation}\frac{1}{\sqrt{n}}\sum^{n}_{i=1}X_{i}\rightarrow \mathcal{N}(0,\sigma^{2}),
\end{equation}
where $\sigma^{2}=E X^{2}_{0}+2\sum^{\infty}_{j=1}E(X_{0}X_{j}).$  In case $\sigma^{2}=0$, $\mathcal{N}(0,0)$ denotes the point mass at the origin. If $X_{0}$ is bounded, the CLT continues to hold if $(\ref{summa})$ is replaced by the condition that $\sum^{\infty}_{k=0}k^{2}(a_{k}+\beta(k))<\infty$.
\end{theorem}

An important tool to derive asymptotic results for weakly dependent data are coupling methods, we will need this method to prove the invariance principle (Proposition \ref{invprin}).


\begin{theorem}[Theorem 3 \cite{BBD}]\label{The3}
Let $(X_{n})_{n\in\N}$ be a $1-$approximating functional with summable constants $(a_{k})_{k\geq 0}$ of an absolutely regular process with mixing rate $(\beta(k))_{k\geq0}$.\;  Then given integers $K,L$ and $N$, we can approximate the sequence of $(K+2L,N)-$blocks $(B_{s})_{s\geq 1}$ by a sequence of independent blocks $(B'_{s})_{s\geq 1}$ with the same marginal distribution in such a way that 
\begin{equation}
P(||B_{s}-B'_{s}||\leq 2\alpha_{L})\geq 1-\beta(K)-2\alpha_{L},
\end{equation}
where
$\alpha_{L}:=\left(2\sum^{\infty}_{l=L}a_{l}\right)^{1/2}.$ 
\end{theorem}

In statistical application, the question of how to estimate $\sigma^2$ is important. In the situation when the observations are a functional of $\alpha-$mixing process, Dehling et al. \cite{DS} propose the estimation of the variance of partial sums of dependent processes by the subsampling estimator
\begin{equation}
\hat{D}_{n}=\frac{1}{[n/l_{n}]}\sqrt{\frac{\pi}{2}}\sum^{[n/l_{n}]}_{i=1}\frac{|\hat{T}_{i}(l_{n})-l_{n}\tilde{U}_{n}|}{\sqrt{l_{n}}}
\end{equation}
with $\hat{T}_{i}(l)=\sum^{il}_{j=(i-1)l+1}F_{n}(X_{j})$ and $\tilde{U}_{n}=\frac{1}{n}\sum^{n}_{j=1}F_{n}(X_{j})$, where $F_{n}(\cdot)$ is the empirical distribution function ({\itshape e.d.f.}).

\begin{theorem}[Theorem 1.2 \cite{DS}]\label{varestima}
Let $(X_{k})_{k\geq 1}$ be a stationary, $1$-approximating functional of an $\alpha-$mixing processes. Suppose that for some $\delta>0$, $E|X_{1}|^{2+\delta}<\infty$, and that the mixing coefficients $(\alpha_{k})_{k\geq 1}$ and the approximation constants $(a_{k})_{k\geq 1}$ satisfy 
\begin{equation}
\sum^{\infty}_{k=1}(\alpha_{k})^{\frac{2}{2+\delta}}<\infty,\quad \sum^{\infty}_{k=1}(a_{k})^{\frac{1+\delta}{2+\delta}}<\infty. 
\end{equation}
In addition, we assume that $F$ is Lipschitz-continuous, that $\alpha_{k}=O(n^{-8})$ and that $a_{m}=O(m^{-12})$. Then, as $n\rightarrow\infty$, $l_n\rightarrow\infty$ and $l_{n}=o(\sqrt{n})$, we have $\hat{D}_{n}\longrightarrow\sigma$ in $L_{2}.$ 
\end{theorem}

To deal with the degenerate kernel $g$, we need to find upper bounds for $E\left(g(X_{i_{1}},X_{j_{1}})g(X_{i_{2}},X_{j_{2}})\right)$, in terms of the maximal distance among the indices. Due to $1\leq i_{1}<i_{2}\leq [n\lambda]$ and $[n\lambda]+1\leq j_{1}<j_{2}\leq n$, w.l.o.g. $i_{1}< i_{2}< j_{1}< j_{2}$.

\begin{lemma}[Proposition 6.1 in \cite{DF}]\label{first}
Let $(X_{n})_{n\geq 1}$ be a $1-$approximating functional with constants $(a_{k})_{k\geq 1}$ of an 
absolutely regular process with mixing coefficients $(\beta(k))_{k\geq 1}$ and let $g(x,y)$ be a $1-$continuous bounded degenerate kernel. Then we have
\begin{equation}
  |E(g(X_{i_{1}},X_{j_{1}})g(X_{i_{2}},X_{j_{2}}))|\leq 4S\phi(a_{[k/3]})+8S^{2}(\sqrt{a_{[k/3]}}+\beta([k/3]))
  \label{bound1}  
\end{equation}
where $S=|\sup_{x,y}g(x,y)|$ and $k=\max\left\{i_{2}-i_{1},j_{1}-i_{2}, j_{2}-j_{1}\right\}$
\end{lemma}

The following two results are useful for proving tightness of a stochastic process. The first one is used to control the fluctuation of maximum. Let $\xi_{1},\ldots,\xi_{n}$ be random variables (stationary or not, independent or not). We denote by  $S_{k}=\xi_{1}+\ldots+\xi_{k}$ ($S_{0}=0$), and put $M_{n}=\max_{0\leq k\leq n}|S_{k}|$.

\begin{theorem}[Theorem 10.2 \cite{B}]\label{Billing}
Suppose that $\beta\geq 0$ and $\alpha>1/2$ and that there exist nonnegative numbers $u_{1},\ldots,u_{n}$ such that for all positive $\lambda$
\begin{equation}
P\left(|S_{j}-S_{i}|\geq \lambda\right)\leq \frac{1}{\lambda^{4\beta}}\left(\sum_{i<l\leq j}u_{l} \right)^{2\alpha},\quad 0\leq i\leq j\leq n\quad,
\end{equation}
then for all positive $\lambda$
\begin{equation}
P\left(M_{n}\geq \lambda\right)\leq \frac{K_{\beta,\alpha}}{\lambda^{4\beta}}\left(\sum_{0<l\leq n}u_{l}   \right)^{2\alpha},
\end{equation}
where $K_{\beta,\alpha}$ is a constant depending only on $\beta$ and $\alpha$.
\end{theorem}

\begin{theorem}[Theorem 8.4 \cite{B}]\label{Billing1}
The sequence $\{Y_{n}\}$, defined by
\begin{equation}
Y_{n}(t)=\frac{1}{\sigma\sqrt{n}}S_{[nt]}+(nt-[nt])\frac{1}{\sigma\sqrt{n}}\xi_{[nt]+1}
\end{equation}
is tight if for each $\epsilon>0$ there exist a $\lambda>1$ and a $n_{0}\in\N$ such that for $n\geq n_{0}$
\begin{equation}
P\left(\max_{i\leq n}|S_{k+i}-S_{k}|\geq \lambda\sigma\sqrt{n}\right)\leq\frac{\epsilon}{\lambda^{2}}.
\end{equation}
\end{theorem}

\bibliographystyle{plain}

\end{document}